\documentclass[10pt,letterpaper]{amsart}

\usepackage{amsmath,amsthm,amsfonts,amssymb}
\usepackage{mathtools} 
\usepackage{mathrsfs} 
\usepackage{stmaryrd}
\usepackage{cancel}
\usepackage[pdftex,bookmarks=true]{hyperref} 
\usepackage[thinlines]{easybmat}
\usepackage{amsrefs}

\newcommand{\mfr}[1]{\ensuremath \mathfrak{#1}}
\newcommand{\mbb}[1]{\ensuremath \mathbb{#1}}
\newcommand{\mbf}[1]{\ensuremath \mathbf{#1}}
\newcommand{\mcl}[1]{\ensuremath \mathcal{#1}}

\newcommand{\mrm}[1]{\ensuremath \mathrm{#1}}

\DeclareMathOperator{\id}{id}
\DeclareMathOperator{\tr}{tr}
\DeclareMathOperator{\rank}{rank}

\DeclareMathOperator{\Rc}{Rc}
\DeclareMathOperator{\Ric}{Ric}
\DeclareMathOperator{\Rm}{Rm}
\DeclareMathOperator{\scal}{scal}

\newcommand{\Ro}{\overset{{}_\circ}{R}}
\newcommand{\opA}{\mbf{L}}

\DeclareMathOperator{\SL}{SL}
\DeclareMathOperator{\SO}{SO}

\newcommand{\smfrac}[2]{{\textstyle{\frac{#1}{#2}}}}
\newcommand{\half}{{\smfrac{1}{2}}}

\numberwithin{equation}{section}

\theoremstyle{plain} 
\newtheorem{thm}[equation]{Theorem}

\newtheorem{prop}[equation]{Proposition}
\newtheorem{conj}[equation]{Conjecture}

\theoremstyle{definition}

\theoremstyle{remark}
\newtheorem{rem}[equation]{Remark}

\begin{document}

\author{Michael Jablonski}
\address{Department of Mathematics, University of Oklahoma}
\email{mjablonski@math.ou.edu}

\author{Peter Petersen}
\address{Department of Mathematics, University of California, Los Angeles}
\email{petersen@math.ucla.edu}

\author[M.~B.~Williams]{Michael Bradford Williams}
\address{Department of Mathematics, University of California, Los Angeles}
\email{mwilliams@math.ucla.edu}


\title{On the linear stability of expanding {R}icci solitons}
\date{\today}

\begin{abstract}
In previous work, the authors studied the linear stability of algebraic Ricci solitons on simply connected solvable Lie groups (solvsolitons), which are stationary solutions of a certain normalization of Ricci flow.  Many examples were shown to be linearly stable, leading to the conjecture that all solvsolitons are linearly stable.  This paper makes progress towards that conjecture, showing that expanding Ricci solitons with bounded curvature (including solvsolitons) are linearly stable after extension by a Gaussian soliton.  As in the previous work, the dynamical stability follows from a generalization of the techniques of Guenther, Isenberg, and Knopf.
\end{abstract}



\maketitle

\section{Introduction}

Given a manifold $M$, $\lambda \in \mbb{R}$, and a vector field $X$, the \textit{curvature-normalized Ricci flow} for metrics on $M$ is
\begin{equation}\label{eq:cnrf}
\partial_t g = -2 \Rc + 2\lambda g + \mcl{L}_X g.
\end{equation}
This differs from ordinary Ricci flow, $\partial_t g = -2 \Rc$, by the action of scaling and diffeomorphisms \cite{GuentherIsenbergKnopf2006}.  The purpose of the normalization is that a Ricci soliton $(M,g_0,\lambda,X)$ satisfying
\begin{equation}\label{eq:ricci-soliton}
\Rc(g_0) = \lambda g_0 + \half \mcl{L}_X g
\end{equation}
is a stationary solution of \eqref{eq:cnrf}.  

As Ricci solitons play a central role in understanding singularity models of Ricci flow, it is therefore interesting to know the dynamical properties of these special solutions, namely, to determine dynamical stability of Ricci solitons relative to \eqref{eq:cnrf}.  For example, a Ricci flow solution $g(t)$ is called \textit{Type-III} if
\[ g(t) \text{ exists } \forall t \geq 0 \qquad \text{ and } \qquad \sup_{M \times [0,\infty)} t \|\Rm(t)\| < \infty. \]
This behavior encompasses a large and important class of Ricci flow solutions.  Indeed, we have the following, due originally to Hamilton. 

\begin{conj}\label{conj:ham}
Type-III solutions asymptotically approach soliton metrics that are (locally) homogeneous.
\end{conj}

Part of the difficulty here is to precisely describe what ``asymptotically approach'' means, but there is recent evidence to support the conjecture, notably the great progress made by Lott in dimension 3 \cites{Lott2007,Lott2010}.  Based on that work, it is reasonable to take the homogeneous solitons to be expanding ($\lambda<0$).  Then, a step towards understanding the conjecture in higher dimensions is, first, to describe the possible expanding homogeneous Ricci solitons, and second, to show that they are dynamically stable under (curvature-normalized) Ricci flow.  However, even the question of understanding homogeneous Ricci solitons is nontrivial. 

\begin{conj}
Every expanding homogeneous Ricci soliton is isometric to a simply-connected solvmanifold. 
\end{conj}

This is the \textit{generalized Alexeevskii conjecture} \cites{LafuenteLauret2013,Jablonski2014}.  We will not address this conjecture further, but all evidence so far supports it.  Up to isometry, all homogeneous Ricci solitons on simply connected solvable Lie groups satisfy the following equation on the Lie algebra level:
\begin{equation}\label{eq:alg-soliton}
\Ric = \lambda \id + D,
\end{equation}
where $\Ric$ is the Ricci endomorphism, $\lambda <0$, and $D$ is a derivation of the Lie algebra.  A metric satisfying this equation is called an \textit{algebraic soliton}; on nilpotent and solvable Lie groups, such metrics are called \textit{nilsolitons} and \textit{solvsolitons}, respectively.  

Regarding dynamical stability, there are many approaches, but we follow the program of Guenther, Isenberg, and Knopf \cites{GuentherIsenbergKnopf2002,GuentherIsenbergKnopf2006}.  After normalizing Ricci flow as in \eqref{eq:cnrf}, one applies the DeTurck trick and linearizes the flow at a fixed point:
\[ \partial_t h = \opA h := \Delta_L h + 2\lambda h + \mcl{L}_X h. \]
where $\Delta_L$ is the Lichnerowicz Laplacian acting on symmetric 2-tensors.  Next one proves the fixed point is strictly linearly stable, that is, that there exists $\epsilon > 0$ such 
\begin{equation}\label{eq:stab-cond}
(\opA h,h) \leq - \epsilon \|h\|^2 
\end{equation}
for all symmetric 2-tensors $h$ taken from some appropriate tensor space.  With this established, dynamical stability is obtained by appealing to a theorem is Simonett \cite{Simonett1995}.  This step is not immediate, especially in the non-compact setting of simply connected solvmanifolds, but work of the third author and Wu addresses this issue \cite{WilliamsWu2013-dynamical}, see Theorem \ref{thm:dyn-stab}.

The goal of this paper is to address the question of linear stability of expanding Ricci solitons, particularly solvsolitons.  In previous work of the authors, many solvsolitons were shown to be linearly stable with respect to \eqref{eq:cnrf}.  The following is \cite{JablonskiEtAl2013-linear}*{Theorem 1.10}.

\begin{thm}\label{thm:stab-all}
The following algebraic solitons are strictly linearly stable with respect to the curvature-normalized Ricci flow:
\begin{enumerate}
\item\label{item1} every nilsoliton of dimension six or less, and every member of a certain one-parameter family of seven-dimensional nilsolitons;
\item every abelian or two-step nilsoliton;
\item every four-dimensional solvsoliton whose nilradical is the three-dimensional Heisenberg algebra;
\item an open set of solvsolitons whose nilradicals are codimension-one and abelian;
\item every solvable Einstein metric whose nilradical is codimension-one and 
\begin{enumerate}
\item found in \ref{item1} and has dimension greater than one, or
\item a generalized Heisenberg algebra, or
\item a free two-step nilpotent algebra.
\end{enumerate}
\item for each $m \geq 2$, an $(8m^2-6m-8)$-dimensional family of negatively-curved Einstein metrics containing the quaternionic hyperbolic space $\mbb{H}H^{m+1}$;
\item an $84$-dimensional family of negatively-curved Einstein metrics containing the Cayley hyperbolic plane $\mbb{C}aH^2$.
\end{enumerate}
Furthermore, any non-nilpotent solvsoliton $\mfr{s} = \mfr{n} \rtimes \mfr{a}$ satisfying condition \eqref{eq:q-cond} and $0 < \dim(\mfr{a}) < \rank(\mfr{n})$ is contained in an open set of stable solvsolitons.
\end{thm}

The multitude of stable solitons---and the lack of any unstable ones---leads to the following.

\begin{conj}\label{conj:stab}
Every algebraic soliton is strictly linearly stable. 
\end{conj}

The main result of this paper makes progress toward positive resolution of this conjecture.  It essentially says that expanding Ricci solitons (not assumed to be algebraic) are stable, possibly after the addition of a Gaussian factor. 

\begin{thm}\label{thm:lin-stab}
Let $(M^n,g_0)$ be an expanding Ricci soliton metric with bounded curvature, satisfying
\[ \Rc(g_0) = \lambda g_0 + \half \mcl{L}_X g_0. \]
Let $(\mbb{R}^k,\delta)$ denote the expanding Gaussian soliton on Euclidean space, scaled so that
\[ 0 = \lambda \delta + \mathrm{Hess}(f), 
\qquad f(y) := -\frac{\lambda}{2} |y|^2. \]
There exists $k\geq 0$ such that the product soliton $(M \times \mbb{R}^k,g_0 + \delta,\lambda,X+\nabla f)$, considered as a fixed point of the curvature-normalized Ricci flow \eqref{eq:cnrf} for metrics on $M \times \mbb{R}^k$, is strictly linearly stable.
\end{thm}

Although the theorems of \cite{WilliamsWu2013-dynamical} apply to a somewhat wider range of spaces, for simplicity we state the following dynamical stability result in terms of algebraic solitons.

\begin{thm}[\cite{WilliamsWu2013-dynamical}*{Theorem 1.4}]\label{thm:dyn-stab}
Suppose $(M,g_0)$ is an algebraic soliton.  The product metric $g_0+\delta$ from Theorem \ref{thm:lin-stab} is dynamically stable in the context of \cite{WilliamsWu2013-dynamical}*{Theorem 1.2}.
\end{thm}

Here is a brief outline of this paper.  In Section \ref{sec:proof}, we give a slight generalization of the criteria for linear stability found in \cite{JablonskiEtAl2013-linear}.  Then we describe Gaussian extensions of expanding Ricci solitons and prove Theorem \ref{thm:lin-stab}.  

In Section \ref{sec:models}, we express Ricci flow on the Gaussian extension as a coupled flow for objects on the original manifold $M$.  To do this, we interpret Gaussian extensions in two different geometric contexts:~as flat vector bundles and as warped products.  The bundle approach is related to work of Lott in which these objects serve as models for expanding Ricci solitons \cite{Lott2007}.  In both cases, the resulting flow is a modified Ricci flow coupled with a modified heat flow.

Finally, in Section \ref{sec:examples}, we provide more evidence for Conjecture \ref{conj:stab} listing several hundred examples of stable algebraic solitons in dimensions seven and eight.

\section{Gaussian extensions of expanding Ricci solitons}\label{sec:proof}

\subsection{Criteria for linear stability}\label{subsec:lin-criteria}

In this section we recall the criteria for linear stability that were derived in Section 2 of \cite{JablonskiEtAl2013-linear}.  Let $(M,g)$ be a Riemannian manifold.  The Lichnerowicz Laplacian, acting on symmetric 2-tensors, has the form
\begin{equation}\label{eq:lich1}
-\Delta_L h = \nabla^*\nabla h - 2 \Ro h + \Ric \circ h + h \circ \Ric.
\end{equation}
Here, $\nabla^*\nabla = -\tr_g \nabla^2 h = -\Delta h$ is the connection Laplacian on 2-tensors.  The second term is the action of the Riemann tensor on symmetric 2-tensors.  In local coordinates, this is
\[ (\Ro h)_{ij} = R_{ipqj} h^{pq}. \]
The last two terms in \eqref{eq:lich1} are an abuse of notation for the action of the Ricci tensor on symmetric 2-tensors.  In local coordinates, this is
\[ (\Ric \circ h + h \circ \Ric)_{ij} = R_i^k h_{kj} + R_j^k h_{ki}. \]
Now, define the quadratic forms that depend on $g$:
\begin{align}
q_x(h) &:= \langle \Ro h + \Ric \circ h,h \rangle\big|_x \label{eq:op-q} \\
Q(h) &:= (\Ro h + \Ric \circ h,h) = \int_M q_x(h) \, d\mu \label{eq:op-Q}
\end{align}
where $(\cdot,\cdot)$ is the $L^2$ inner product on symmetric $2$-tensors.

\begin{prop}\label{prop:stab-cond}
Let $(M,g,\lambda,X)$ be a Ricci soliton.  The following condition implies that $g$ is strictly linearly stable:
\[ Q(h) < \frac{1}{2} \int_M \mrm{div}(X) |h|^2 \, d\mu \]
for all $h$ where these expressions are defined.
\end{prop}

\begin{proof}
This follows directly from the proof of Proposition 2.9 in \cite{JablonskiEtAl2013-linear}.  The only difference is that if we do not assume that $X$ has constant divergence, then the simplification in equation (2.7) there does not carry through.
\end{proof}

\subsection{Proof of Theorem \ref{thm:lin-stab}}

Now we can prove \eqref{thm:lin-stab}.  We are given a Ricci soliton $(M^n,g_0,\lambda,X)$.  For some $k$, recall that the Gaussian soliton on $\mbb{R}^k$ is the usual Euclidean metric $\delta$ interpreted as a gradient Ricci soliton:
\[  0 = \lambda \delta + \mrm{Hess}(f), \]
where $f(y) := - \frac{\lambda}{2} |y|^2$ and $|y|^2$ is the square of the distance to the origin.  We can form the \textit{Gaussian extension} of $g$, which is the Ricci soliton 
\[(M \times \mbb{R}^k,g_0+\delta,\lambda,X+\nabla f). \]
It is easy to check that this is indeed a soliton, and it is expanding/steady/shrinking if and only if $(M,g_0,\lambda,X)$ is.

Assume that $g_0$ is expanding (so $\lambda < 0$), and consider the curvature-normalized Ricci flow for metrics $\mbf{g}$ on the product $M \times \mbb{R}^k$:
\[ \partial_t \mbf{g} = -2 \Rc(\mbf{g}) + 2 \lambda \mbf{g} + \mcl{L}_{X+\nabla f} \mbf{g}. \]
After applying the DeTurck trick, the linearization of the right side of the above equation is
\[ \opA h := \Delta_L h + 2\lambda h + \mcl{L}_{X + \nabla f} h \]
acting on symmetric 2-tensors on $M \times \mbb{R}^k$.  See \cite{GuentherIsenbergKnopf2006} for details.

For linear stability, by Proposition \ref{prop:stab-cond} we need to show that 
\[ Q(h) - \frac{1}{2} \int_{M \times \mbb{R}^k} \mrm{div}(X+\nabla f) |h|^2 \, d\mu \stackrel{?}{<} 0, \]
where all quantities are computed with respect to $g_0+\delta$.  At any point $(x,y) \in M \times \mbb{R}^k$, we estimate $q_{(x,y)}(h)$ using that both $\Ro$ and $\Ric$ split relative to the product metric, that $\delta$ is flat, that $g_0$ has bounded curvature, and the Cauchy inequality.  In local coordinates $(x^i)$ on $M$, 
\begin{align*}
|q_{(x,y)}(h)|
&= \left| \sum_{ijk\ell} R_{ijk\ell} h_{i\ell} h_{jk} 
 + \sum_{ijk} R_{ij} h_{ik} h_{jk} \right| \\
&= \frac{1}{2} \sum_{ijk\ell} |R_{ijk\ell}|[ h_{i\ell}^2 + h_{jk}^2] 
 + \frac{1}{2} \sum_{ij} |R_{ijk}|[ h_{ik}^2 + h_{jk}^2] \\
&\leq C_1 |h|^2
\end{align*}
where $C_1$ is a constant depending only on $(M,g_0)$.  It is also clear that
\[ \mrm{div}(X+\nabla f) = 
\mrm{div}(X) + \mrm{div}(\nabla f) =
\mrm{div}(X) - \lambda k, \]
and tracing \eqref{eq:ricci-soliton} gives
\[ |\mrm{div}(X)| \leq |\scal(g)| + n|\lambda| \leq 2 C_2, \]
where $C_2$ is a constant depending only on $(M,g_0)$.

Combining these estimates, we see that
\begin{align*}
Q(h) - \frac{1}{2} \int_{M \times \mbb{R}^k} \mrm{div}(X+\nabla f) |h|^2 \, d\mu 
&\leq \int_{M \times \mbb{R}^k} \left[ C_1 + C_2 + \frac{1}{2} \lambda k \right] |h|^2 \, d\mu \\
&= \left[ C_1 + C_2 + \frac{1}{2} \lambda k \right] \|h\|^2.
\end{align*}
Since $\lambda < 0$ and the constants do not depend on $k$, we are free to choose $k$ large enough so that the quantity in brackets is, say, less than $-1$.  Thus, $g_0+\delta$ is strictly linearly stable.  If $(M,g_0,\lambda,X)$ is already linearly stable, then of course we can choose $k=0.$  This completes the proof.

\section{Geometric models for stable expanding solitons}\label{sec:models}

Given an expanding Ricci soliton $(M,g_0)$ with bounded curvature, Theorem \ref{thm:lin-stab} implies that $g_0+\delta$ is stable under curvature-normalized Ricci flow on $M \times \mbb{R}^k$ for some $k \geq 0$, but what can we conclude about $g_0$ itself?  Under which flow(s) is $g_0$ stable?  We answer this question by considering two geometric contexts where the curvature-normalized Ricci flow for metrics on the product $M \times\mbb{R}^k$ can be written as a coupled flow for objects on $M$ only.

\subsection{Flat vector bundles}

As models for expanding homogeneous Ricci solitons, Lott studied a certain class of vector bundles \cite{Lott2007}.  Let $N$ be an $\mbb{R}^k$-vector bundle with flat connection, flat metrics $G$ on the fibers, and Riemannian base $(M^n,g)$.  Assume that the connection preserves fiberwise volume forms.  The inner products on the fibers can be thought of as a map $G : (M,g) \rightarrow (\SL(k,\mbb{R})/\SO(k),\theta)$, where $\theta$ is the natural metric induced by $G$:
\[ \theta_G(X,Y) = \tr(G^{-1} X G^{-1} Y). \]

There is a natural corresponding metric $\mbf{g}$ on $N$, which is locally a product, and in this class of metrics Ricci flow $\partial_t \mbf{g} = -2 \Rc[\mbf{g}]$ is the coupled flow
\begin{align*}
\partial_t g &= -2 \Rc + \frac{1}{4} \nabla G \otimes \nabla G \\
\partial_t G &=  \tau_{g,\theta} G
\end{align*}
where $\tau_{g,\theta} G$ is the harmonic map Laplacian of $G$ and $\nabla G \otimes \nabla G = G^*\theta$; see \cite{Williams2013-hrf}*{Section 2.1}.  This type of coupled flow is called harmonic-Ricci flow, or Ricci flow coupled with harmonic map flow, and has been studied by various authors, e.g., \cites{List2008,Muller2012}.

In our case, the product $(M \times \mbb{R}^k,g_0+\delta)$ is trivially a bundle of the form just described (where $G = \delta$ is constant), and we want to describe how curvature-normalized Ricci flow
\[ \partial_t \mbf{g} = -2 \Rc(\mbf{g}) + 2 \lambda \mbf{g} + \mcl{L}_{X+\nabla f} \mbf{g}. \]
can be written as a coupled flow.  We need to understand the ``new'' terms $2 \lambda \mbf{g} + \mcl{L}_{X+\nabla f} \mbf{g}$.  From this, the first extra term is
\[ 2\lambda g\mbf{g} = 2\lambda g + 2\lambda G, \]
at least locally.  The Lie derivative term is
\[ \mcl{L}_{X+\nabla f} \mbf{g} = \mcl{L}_X g + 2 \, \mrm{Hess}_G(f). \]
This means that the flow is
\begin{equation}\label{eq:rf-bundle}
\begin{aligned}
\partial_t g &= -2 \Rc + 2 \lambda g + \mcl{L}_X g + \frac{1}{4} \nabla G \otimes \nabla G \\
\partial_t G &=  \tau_{g,\theta} G + 2 \lambda G + 2 \, \mrm{Hess}_G(f)
\end{aligned}
\end{equation}

\subsection{Warped products}

Another model for certain Ricci solitons is warped products, see \cites{HeEtAl2013,Lafuente2014}.  Let $M \times N$ be a product manifold with warped product metric
\[ \mbf{g} := g + e^{2\phi} \gamma, \]
where $g$ is a metric on $M$, $\phi \in C^\infty(M)$, and $\gamma$ is a metric on $N$.  In this class of metrics, Ricci flow $\partial_t \mbf{g} = -2 \Rc[\mbf{g}]$ is the coupled flow
\begin{align*}
\partial_t g    &= -2 \Rc + 2k \, \mrm{Hess}(\phi) + 2k \, d\phi \otimes d\phi \\
\partial_t \phi &= \Delta \phi + k |d\phi|^2
\end{align*}
See \cite{Williams2013-systems}*{Section 3}.

In our case, we take warped products $(M \times \mbb{R}^k,g+e^{2\phi} \delta)$, so that $g_0 + \delta$ is such a metric (where $\phi$ vanishes).  We want to describe how curvature-normalized Ricci flow
\[ \partial_t \mbf{g} = -2 \Rc(\mbf{g}) + 2 \lambda \mbf{g} + \mcl{L}_{X+\nabla f} \mbf{g} \]
can be written as a coupled flow.  

We again need to understand the ``new'' terms $2 \lambda \mbf{g} + \mcl{L}_{X+\nabla f} \mbf{g}$.  The first extra term is obvious:
\[ 2\lambda \mbf{g} = 2\lambda g + 2\lambda e^{2\phi} \delta. \]
The Lie derivative term is
\[ \mcl{L}_{X +\nabla f} \mbf{g} 
= \mcl{L}_X g + 2 e^{2\phi} \mrm{Hess}(f) 
= \mcl{L}_X g - 2\lambda e^{2\phi} \delta. \]
Now, as in \cite{Williams2013-systems}*{Proposition 3.2}, curvature-normalized Ricci flow becomes
\begin{align*}
\partial_t g    &= -2 \Rc + 2k \, \mrm{Hess}(\phi) + 2k \, d\phi \otimes d\phi + 2\lambda g + \mcl{L}_X g \\
\partial_t \phi &= \Delta \phi + k |d\phi|^2
\end{align*}
This can be simplified.  Pulling back by diffeomorphisms generated by the vector field $-k \nabla \phi$ results in the addition of the following Lie derivatives to the above equations:
\[ \mcl{L}_{-k \nabla \phi} g = -2k \, \mrm{Hess}(\phi), \qquad
 \mcl{L}_{-k\nabla \phi} \phi = -k |d\phi|^2, \]
so the modified flow is
\begin{equation}\label{eq:rf-wp}
\begin{aligned}
\partial_t g    &= -2 \Rc + 2k \, d\phi \otimes d\phi + 2\lambda g + \mcl{L}_X g \\
\partial_t \phi &= \Delta \phi
\end{aligned}
\end{equation}

\begin{rem}
By Theorem \ref{thm:lin-stab}, we know that the metric $g_0+\delta$ is a stable fixed point of curvature-normalized Ricci flow for some $k\geq 0$.  This implies that $g_0+\delta$ is a stable fixed point the systems \eqref{eq:rf-bundle} and \eqref{eq:rf-wp}.  Thus, stability of these systems is obtained from the opposite direction of \cite{Williams2013-systems}, which derives stability of similar coupled flows given information about the individual flows.
\end{rem}

\section{Algebraic solitons}\label{sec:examples}


We saw in Theorem \ref{thm:stab-all} that all nilsolitons of dimension $6$ or less are strictly linearly stable.  This was proven by verifying the following algebraic condition for linear stability:
\begin{equation}\label{eq:q-cond}
q(h) < \frac{1}{2} \tr D |h|^2
\end{equation}
for all $2$-tensors on the Lie algebra.  (Recall that $q$ was defined in \eqref{eq:op-q}.)  In this section we expand this list to include more examples in dimensions $7$ and $8$.  As in \cite{JablonskiEtAl2013-linear}, the condition is verified using a computer algebra system.

Kadioglu and Payne recently classified all $7$- and $8$-dimensional nilsolitons with simple pre-Einstein derivation \cite{KadiogluPayne2013}.  As in \cite{JablonskiEtAl2013-linear}, we show that these examples, and their 1-dimensional solvable Einstein extensions, are strictly linearly stable.  See Tables 1, 2, 3, 4.  We number the examples as in Tables III and IV in \cite{KadiogluPayne2013}.

\begin{table}
\begin{tabular}{|cccccc|cc|}
\hline
 \multicolumn{6}{|c|}{Nilsoliton}     & \multicolumn{2}{c|}{Solvable extension} \\    
\# &  step  &  $\lambda$ &  $\tr D$  &  max $q$ &  $\overset{?}{<} \half \tr D$ & max $\Ro$ &  $\overset{?}{<} -\lambda$ \\
\hline
 1 & 6 & -17 & 100 & 7.813 & \checkmark & 11.920 & \checkmark \\
 2 & 6 & -18.5 & 112 & 9.455 & \checkmark & 13.356 & \checkmark \\
 3 & 3 & -14.5 & 87.5 & 6.181 & \checkmark & 8.775 & \checkmark \\
 4 & 5 & -176.5 & 1012.5 & 86.000 & \checkmark & 128.966 & \checkmark \\
 5 & 4 & -13.5 & 80 & 6.735 & \checkmark & 10.560 & \checkmark \\
 6 & 5 & -9 & 53 & 4.086 & \checkmark & 6.482 & \checkmark \\
 7 & 3 & -18.5 & 112 & 10.473 & \checkmark & 13.088 & \checkmark \\
 8 & 3 & -5.5 & 33.5 & 2.836 & \checkmark & 4.069 & \checkmark \\
 9 & 4 & -7.5 & 45.5 & 4.064 & \checkmark & 6.128 & \checkmark \\
10 & 3 & -5.5 & 33.5 & 2.827 & \checkmark & 4.137 & \checkmark \\
11 & 3 & -6 & 37 & 2.585 & \checkmark & 4.792 & \checkmark \\
12 & 3 & -6.5 & 40 & 2.336 & \checkmark & 4.393 & \checkmark \\
13 & 3 & -6.5 & 40 & 2.513 & \checkmark & 4.642 & \checkmark \\
14 & 3 & -8 & 49 & 2.780 & \checkmark & 4.748 & \checkmark \\
15 & 4 & -9.5 & 56.5 & 4.808 & \checkmark & 7.678 & \checkmark \\
16 & 4 & -18.5 & 112 & 7.694 & \checkmark & 12.115 & \checkmark \\
17 & 5 & -13.5 & 80 & 7.399 & \checkmark & 10.940 & \checkmark \\
18 & 3 & -7 & 43 & 2.910 & \checkmark & 4.872 & \checkmark \\
19 & 3 & -10.5 & 62.5 & 4.707 & \checkmark & 8.502 & \checkmark \\
20 & 4 & -15.5 & 94 & 8.066 & \checkmark & 12.165 & \checkmark \\
21 & 3 & -11.5 & 70 & 3.654 & \checkmark & 6.770 & \checkmark \\
22 & 3 & -11.5 & 70 & 5.648 & \checkmark & 8.405 & \checkmark \\
23 & 4 & -16.5 & 95 & 8.912 & \checkmark & 13.739 & \checkmark \\
24 & 3 & -18.5 & 112 & 7.914 & \checkmark & 11.433 & \checkmark \\
25 & 5 & -34 & 203 & 15.663 & \checkmark & 24.535 & \checkmark \\
26 & 4 & -14.5 & 113.5 & 10.588 & \checkmark & 15.823 & \checkmark \\
27 & 3 & -18.5 & 112 & 9.029 & \checkmark & 11.746 & \checkmark \\
28 & 3 & -23.5 & 137 & 10.488 & \checkmark & 19.260 & \checkmark \\
29 & 4 & -34 & 203 & 14.334 & \checkmark & 23.417 & \checkmark \\
\hline 
\end{tabular}
\caption{Linear stability of 7-dimensional nilsolitons with simple pre-Einstein derivation, and their 1-dimensional solvable Einstein extensions.}
\label{table:solitons7}
\normalsize
\end{table}

\begin{table}
\begin{tabular}{|cccccc|cc|}
\hline
 \multicolumn{6}{|c|}{Nilsoliton}     & \multicolumn{2}{c|}{Solvable extension} \\    
\# &  step  &  $\lambda$ &  $\tr D$  &  max $q$ &  $\overset{?}{<} \half \tr D$ & max $\Ro$ &  $\overset{?}{<} -\lambda$ \\
\hline
1 & 7 & -79.5 & 541.5 & 39.522 & \checkmark & 59.029 & \checkmark \\
2 & 7 & -14.5 & 102 & 6.721 & \checkmark & 9.467 & \checkmark \\
3 & 4 & -8 & 57 & 3.423 & \checkmark & 5.375 & \checkmark \\
4 & 5 & -11.5 & 81.5 & 6.100 & \checkmark & 8.736 & \checkmark \\
5 & 6 & -18.5 & 130.5 & 9.455 & \checkmark & 13.340 & \checkmark \\
6 & 3 & -200.5 & 1404.5 & 100.132 & \checkmark & 118.767 & \checkmark \\
7 & 4 & -256.5 & 1740.5 & 141.401 & \checkmark & 182.632 & \checkmark \\
8 & 4 & -340.5 & 2244.5 & 187.360 & \checkmark & 264.682 & \checkmark \\
9 & 5 & -9 & 62 & 4.086 & \checkmark & 6.416 & \checkmark \\
10 & 3 & -17 & 117 & 7.826 & \checkmark & 12.670 & \checkmark \\
11 & 3 & -5.5 & 39.5 & 2.241 & \checkmark & 3.056 & \checkmark \\
12 & 4 & -22 & 153 & 8.077 & \checkmark & 13.350 & \checkmark \\
13 & 5 & -10.5 & 73 & 4.304 & \checkmark & 6.817 & \checkmark \\
14 & 6 & -14.5 & 99 & 8.158 & \checkmark & 11.814 & \checkmark \\
15 & 3 & -6.5 & 46 & 2.268 & \checkmark & 3.871 & \checkmark \\
16 & 3 & -18.5 & 130.5 & 7.888 & \checkmark & 12.987 & \checkmark \\
17 & 3 & -6.5 & 46.5 & 2.513 & \checkmark & 4.582 & \checkmark \\
18 & 4 & -10.5 & 69.5 & 5.745 & \checkmark & 9.044 & \checkmark \\
19 & 4 & -9 & 62 & 3.902 & \checkmark & 6.580 & \checkmark \\
20 & 4 & -10.5 & 69.5 & 5.980 & \checkmark & 8.810 & \checkmark \\
21 & 3 & -9.5 & 66 & 3.807 & \checkmark & 7.022 & \checkmark \\
22 & 4 & -9.5 & 66 & 4.808 & \checkmark & 7.586 & \checkmark \\
23 & 5 & -12 & 82 & 5.710 & \checkmark & 8.631 & \checkmark \\
24 & 6 & -17 & 117 & 7.813 & \checkmark & 11.801 & \checkmark \\
25 & 3 & -6.5 & 46 & 3.302 & \checkmark & 4.321 & \checkmark \\
26 & 3 & -7.5 & 51.5 & 3.364 & \checkmark & 5.732 & \checkmark \\
27 & 3 & -7 & 50 & 2.910 & \checkmark & 4.800 & \checkmark \\
28 & 3 & -9 & 62 & 4.479 & \checkmark & 6.821 & \checkmark \\
29 & 4 & -13.5 & 93.5 & 6.735 & \checkmark & 10.421 & \checkmark \\
30 & 4 & -16 & 109 & 8.423 & \checkmark & 12.639 & \checkmark \\
31 & 4 & -19.5 & 133 & 10.451 & \checkmark & 15.839 & \checkmark \\
32 & 3 & -8 & 57 & 3.390 & \checkmark & 4.791 & \checkmark \\
33 & 3 & -8.5 & 60.5 & 3.390 & \checkmark & 4.690 & \checkmark \\
34 & 3 & -8 & 57 & 4.386 & \checkmark & 5.075 & \checkmark \\
35 & 3 & -9.5 & 66 & 4.738 & \checkmark & 7.126 & \checkmark \\
36 & 3 & -9.5 & 66 & 4.557 & \checkmark & 6.222 & \checkmark \\
37 & 4 & -12 & 83 & 6.286 & \checkmark & 9.650 & \checkmark \\
38 & 5 & -23.5 & 160.5 & 13.433 & \checkmark & 19.596 & \checkmark \\
39 & 3 & -9 & 62 & 4.535 & \checkmark & 6.553 & \checkmark \\
40 & 4 & -13.5 & 89 & 7.070 & \checkmark & 10.899 & \checkmark \\
\hline 
\end{tabular}
\caption{Linear stability of 8-dimensional nilsolitons with simple pre-Einstein derivation, and their 1-dimensional solvable Einstein extensions.}
\label{table:solitons8-1}
\normalsize
\end{table}

\begin{table}
\begin{tabular}{|cccccc|cc|}
\hline
 \multicolumn{6}{|c|}{Nilsoliton}     & \multicolumn{2}{c|}{Solvable extension} \\    
\# &  step  &  $\lambda$ &  $\tr D$  &  max $q$ &  $\overset{?}{<} \half \tr D$ & max $\Ro$ &  $\overset{?}{<} -\lambda$ \\
\hline
41 & 4 & -13.5 & 89 & 7.527 & \checkmark & 10.458 & \checkmark \\
42 & 3 & -10.5 & 73 & 5.467 & \checkmark & 8.270 & \checkmark \\
43 & 4 & -15.5 & 109.5 & 8.066 & \checkmark & 12.022 & \checkmark \\
44 & 3 & -11.5 & 81.5 & 3.654 & \checkmark & 6.680 & \checkmark \\
45 & 3 & -11.5 & 81.5 & 4.086 & \checkmark & 7.601 & \checkmark \\
46 & 3 & -13.5 & 93.5 & 6.229 & \checkmark & 10.239 & \checkmark \\
47 & 3 & -12 & 82 & 5.602 & \checkmark & 7.707 & \checkmark \\
48 & 4 & -14.5 & 97 & 7.409 & \checkmark & 10.136 & \checkmark \\
49 & 3 & -11.5 & 81.5 & 5.650 & \checkmark & 7.987 & \checkmark \\
50 & 3 & -12 & 85 & 5.645 & \checkmark & 8.091 & \checkmark \\
51 & 3 & -14.5 & 102 & 6.181 & \checkmark & 8.613 & \checkmark \\
52 & 3 & -14.5 & 102 & 6.134 & \checkmark & 8.162 & \checkmark \\
53 & 3 & -15.5 & 109 & 7.349 & \checkmark & 10.669 & \checkmark \\
54 & 4 & -18.5 & 130.5 & 7.694 & \checkmark & 11.925 & \checkmark \\
55 & 3 & -11.5 & 81.5 & 6.123 & \checkmark & 8.214 & \checkmark \\
56 & 3 & -13.5 & 93.5 & 5.697 & \checkmark & 9.333 & \checkmark \\
57 & 3 & -13.5 & 93.5 & 6.532 & \checkmark & 10.383 & \checkmark \\
58 & 3 & -18.5 & 130.5 & 7.914 & \checkmark & 11.240 & \checkmark \\
59 & 3 & -18.5 & 130.5 & 7.808 & \checkmark & 11.778 & \checkmark \\
60 & 3 & -14.5 & 99 & 6.512 & \checkmark & 11.610 & \checkmark \\
61 & 3 & -16.5 & 111.5 & 6.469 & \checkmark & 12.251 & \checkmark \\
62 & 3 & -15.5 & 109.5 & 5.946 & \checkmark & 11.096 & \checkmark \\
63 & 3 & -14.5 & 102 & 7.694 & \checkmark & 8.672 & \checkmark \\
64 & 4 & -29.5 & 205.5 & 10.444 & \checkmark & 16.705 & \checkmark \\
65 & 4 & -19.5 & 133 & 10.588 & \checkmark & 15.556 & \checkmark \\
66 & 3 & -17 & 117 & 7.887 & \checkmark & 11.407 & \checkmark \\
67 & 5 & -146.5 & 999 & 77.790 & \checkmark & 116.680 & \checkmark \\
68 & 3 & -18.5 & 130.5 & 5.702 & \checkmark & 11.013 & \checkmark \\
69 & 3 & -19 & 134 & 5.448 & \checkmark & 11.009 & \checkmark \\
70 & 3 & -18.5 & 130.5 & 9.029 & \checkmark & 11.509 & \checkmark \\
71 & 3 & -18.5 & 130.5 & 8.980 & \checkmark & 12.209 & \checkmark \\
72 & 4 & -28.5 & 188.5 & 14.569 & \checkmark & 23.386 & \checkmark \\
73 & 3 & -18.5 & 130.5 & 10.473 & \checkmark & 12.848 & \checkmark \\
74 & 3 & -19 & 134 & 9.442 & \checkmark & 12.090 & \checkmark \\
75 & 4 & -34 & 237 & 11.817 & \checkmark & 20.758 & \checkmark \\
76 & 3 & -19.5 & 133 & 11.861 & \checkmark & 14.439 & \checkmark \\
77 & 4 & -24 & 163 & 13.468 & \checkmark & 18.668 & \checkmark \\
78 & 4 & -29.5 & 205.5 & 11.989 & \checkmark & 19.023 & \checkmark \\
79 & 3 & -23.5 & 160.5 & 10.712 & \checkmark & 19.585 & \checkmark \\
80 & 3 & -23.5 & 160.5 & 13.329 & \checkmark & 17.566 & \checkmark \\
\hline 
\end{tabular}
\caption{Linear stability of 8-dimensional nilsolitons with simple pre-Einstein derivation, and their 1-dimensional solvable Einstein extensions, continued.}
\label{table:solitons8-2}
\normalsize
\end{table}

\begin{table}
\begin{tabular}{|cccccc|cc|}
\hline
 \multicolumn{6}{|c|}{Nilsoliton}     & \multicolumn{2}{c|}{Solvable extension} \\    
\# &  step  &  $\lambda$ &  $\tr D$  &  max $q$ &  $\overset{?}{<} \half \tr D$ & max $\Ro$ &  $\overset{?}{<} -\lambda$ \\
\hline
81 & 4 & -41.5 & 281.5 & 20.256 & \checkmark & 30.823 & \checkmark \\
82 & 3 & -23.5 & 160.5 & 10.488 & \checkmark & 19.057 & \checkmark \\
83 & 4 & -28.5 & 188.5 & 14.720 & \checkmark & 22.615 & \checkmark \\
84 & 4 & -34 & 237 & 14.334 & \checkmark & 23.087 & \checkmark \\
85 & 4 & -40.5 & 278 & 20.612 & \checkmark & 31.044 & \checkmark \\
86 & 4 & -34.5 & 233 & 18.965 & \checkmark & 27.786 & \checkmark \\
87 & 4 & -40.5 & 278 & 20.790 & \checkmark & 31.394 & \checkmark \\
88 & 4 & -56 & 393 & 22.751 & \checkmark & 34.362 & \checkmark \\
89 & 5 & -61.5 & 407 & 32.577 & \checkmark & 51.080 & \checkmark \\
90 & 3 & -31.5 & 221.5 & 14.690 & \checkmark & 21.833 & \checkmark \\
91 & 3 & -39.5 & 277 & 20.431 & \checkmark & 26.032 & \checkmark \\
92 & 4 & -61.5 & 407 & 32.462 & \checkmark & 51.602 & \checkmark \\
93 & 4 & -55.5 & 373 & 23.128 & \checkmark & 39.407 & \checkmark \\
94 & 4 & -53.5 & 363 & 22.080 & \checkmark & 37.766 & \checkmark \\
95 & 6 & -140.5 & 957 & 69.987 & \checkmark & 104.704 & \checkmark \\
96 & 3 & -56 & 393 & 25.881 & \checkmark & 33.407 & \checkmark \\
97 & 3 & -146.5 & 999 & 64.925 & \checkmark & 117.734 & \checkmark \\
98 & 4 & -113.5 & 786 & 54.418 & \checkmark & 84.688 & \checkmark \\
99 & 5 & -140.5 & 957 & 74.696 & \checkmark & 110.293 & \checkmark \\
100 & 4 & -113.5 & 786 & 60.648 & \checkmark & 80.089 & \checkmark \\
101 & 3 & -113.5 & 786 & 43.757 & \checkmark & 78.935 & \checkmark \\
102 & 4 & -148.5 & 1021 & 79.862 & \checkmark & 109.395 & \checkmark \\
103 & 5 & -176.5 & 1189 & 86.000 & \checkmark & 127.249 & \checkmark \\
104 & 3 & -113.5 & 786 & 47.763 & \checkmark & 76.637 & \checkmark \\
105 & 4 & -140.5 & 957 & 71.563 & \checkmark & 104.468 & \checkmark \\
106 & 4 & -146.5 & 999 & 78.439 & \checkmark & 112.347 & \checkmark \\
107 & 4 & -176.5 & 1189 & 79.759 & \checkmark & 119.755 & \checkmark \\
108 & 3 & -140.5 & 957 & 83.414 & \checkmark & 99.975 & \checkmark \\
109 & 3 & -146.5 & 999 & 64.044 & \checkmark & 115.905 & \checkmark \\
\hline 
\end{tabular}
\caption{Linear stability of 8-dimensional nilsolitons with simple pre-Einstein derivation, and their 1-dimensional solvable Einstein extensions, continued.}
\label{table:solitons8-3}
\normalsize
\end{table}




\end{document}